\documentclass[10pt]{article}
\usepackage[margin=1in]{geometry}
\usepackage{amsmath, amssymb, amsthm}
\usepackage{enumerate, enumitem}
\usepackage[mathscr]{euscript}
\usepackage{fancyhdr, graphicx, proof, comment, multicol}
\usepackage{changepage}
\usepackage[none]{hyphenat} % This command prevents hyphenation of words
\binoppenalty=\maxdimen % This command and the next prevent in-line equation breaks
\relpenalty=\maxdimen
\usepackage{microtype} % Modifies spacing between letters and words
\usepackage{mathpazo} % Modifies font. Optional package.
\usepackage{mdframed} % Required for boxed problems.
\usepackage{parskip} % Left justifies new paragraphs.

\newcommand{\R}{\mathbb{R}}
\newcommand{\C}{\mathbb{C}}

\newcommand{\N}{\mathbb{N}}

\newcommand{\st}{\text{ : }}

\newcommand{\X}{X}  
\newcommand{\XX}{X^*}
\newcommand{\BX}{\mathcal{B}(X)}
\newcommand{\K}{K}  
\newcommand{\KK}{K^*}

\title{The Basic Reproduction Number for Bounded Linear Operators on Ordered Banach Spaces}
\author{Zachary Gregg, Patrick De Leenheer }
\date{January 2026}

\begin{document}

\maketitle

\begin{abstract}
A basic reproduction number, $R_0$, is a concept encountered frequently in the study of ecological and epidemiological models. 
It is routinely used to determine the stability of an extinction or a disease-free fixed point or steady state. It is well-known that for linear models described by non-negative matrices, the spectral radius of the matrix is always contained in an interval with endpoints $1$ and $R_0$.
Here we extend these results to more general cone-preserving bounded linear operators acting on Banach spaces. 
\end{abstract}

\section{ Introduction}

A basic reproduction number, denoted $R_0$, is a tool used in ecological and epidemiological modeling as a marker for the stability of an extinction or disease-free fixed point or steady state; see \cite{allen,cushing-odo,cushing,pdl,diekmann,vdd} for examples. The mathematical underpinnings of basic reproduction numbers can be traced back at least to the 1960's in work by Varga \cite{varga} and Vandergraft \cite{vandergraft}. In essence, for a linear operator whose spectral radius is difficult or even impossible to compute, a basic reproduction number serves as a proxy for that spectral radius, and is itself defined as the spectral radius of an associated linear operator. To make this statement precise, we review 
a result in \cite{li-schneider} by Li and Schneider who define basic reproduction numbers for discrete-time dynamical systems of the form $x_{n+1} = Ax_{n}$, where $x_n$ is a non-negative vector (entry-wise) representing the state of the system at time step $n$, and $A$ is a non-negative matrix (entry-wise). The spectral radius of the matrix $A$ is denoted by $r(A)$, and defined as 
	$$
	r(A) := \max\{ |\lambda| \st \lambda \text{ is an eigenvalue of } A \}.
	$$

\newtheorem{theorem}{Theorem}[section]
\newtheorem{lemma}[theorem]{Lemma}
\newtheorem{corollary}[theorem]{Corollary}
\begin{theorem}\label{lischneidertrich}
(Theorem 3.3 in \cite{li-schneider})

% we need to define r(A) before this statement
Let $A$ be a nonnegative matrix (entry-wise). Suppose $A$ splits as $A=T+F$, where $T$ and $F$ are nonnegative matrices, and $r(T)<1$. Define the basic reproduction number $R_0:= r(F(I-T)^{-1})$. If $R_0>0$, then
$r(T+\frac{1}{R_0}F)=1$, and exactly one of the following cases holds: 

1. $R_0 = r(A) = 1,$

2. $R_0 \leq r(A) < 1,$

3. $R_0 \geq r(A) > 1.$
\end{theorem}

Since $r(A)$being less than or greater than 1 determines the stability of the zero fixed point under the matrix $A$, this shows that $R_0$ can also be used to determine the stability of the zero fixed point. This theorem gives a general version of the basic reproduction number that can be applied to the linearization of the zero steady state for any nonnegative discrete time dynamical system in $R^n$. Note that the basic reproduction number depends on the specific splitting, $A=T+F$, used. Some precautions about and examples using $R_0$ can be found in \cite{pdl}.

% Do we include the classical Leslie example here? 

%\textbf{Example - Leslie Matrix} 

%One of the simpler applications to illustrate this Theorem is population growth using a Leselie Matrix model.  

% Check this statement for accuracy

%\subsection{ Thieme Unordered Trichotomy } 

In 2009, Thieme gave an extension of this theorem that combines the theory of resolvent positive operators\cite{ArendtWolfgang1987RPO} \cite{VoigtJurgen1989Orpo}, completely montone functions \cite{Korenblumetal}\cite{RThiemeHorst1998Rorp}, and superconvex (logconvex) operator families \cite{KINGMANJ.F.C.1961ACPO}\cite{KatoTosio1982Sots}. The definitions used in the statement of the theorem will be introduced in the following sections, but we state this now to give a comparison with the previous theorem.

\begin{theorem}\label{thiemetrich} (Theorem 3.10 in \cite{THIEMEHORSTR.2009SBAR})
Let $X$ be a Banach space with a generating and normal cone $K$. Suppose $A$ is a bounded linear operator that preserves $K$, and $A$ splits as $A=T+F$, where $T$ and $F$ are both $K-$preserving linear operators with $r(T)<1$. Define $R_0:= r(F(I-T)^{-1})$. Then 

1. $R_0 = 1 \iff  r(A) = 1 $

2. $R_0<1 \iff  r(A) <1 $

3.$ R_0 > 1 \iff  r(A) >1 $

\end{theorem}

% If we need extra discussion here, we could say that the positive orthant in $R^n$ is always generating since it has interior. Furthermore, every finite dimensional cone is normal (Cite Nussbaum and Lemmons for this fact, from their Perron Frobenius book). 

Note that, unlike in the theorem by Li and Schneider, the previous theorem does not guarantee that $R_0$ is a bound for the spectral radius of the operator $A$. Our main goal is to show that we can recover $R_0$ as a bound for $r(A)$ in this Banach space setting without any further assumptions.

In order to prove this extension, we will need to make use of another theorem by Thieme \cite{RThiemeHorst1998Rorp} that shows the map $\lambda \to r(F(\lambda I-T)^{-1})$ on $(r(T),\infty)$ is a non-increasing and convex function that will achieve $r(F(\lambda^* I-T)^{-1})< 1$ for some $\lambda^*> r(T)$. The following section will give the background needed in order to state this theorem. 

\section { Background }

\subsection{Cones}

The following generalizes the notion of nonnegativity to the Banach space setting. Throughout let $\X\neq\{0\}$ be a real Banach space, $\XX$ the dual space of bounded (continuous) linear functionals on $\X$, and $\BX$ the Banach space of continuous linear operators from $\X$ to $\X$. 

%\textbf{Definitions}

%[Talk about Krein Rutman as soon as possible in effort to tie back to introductory discussion. This will have to wait until the spectral theory portion. Then talk about generalizing to guaranteeing spectral radius to be in the spectrum and the pole of the resolvent stuff. Go pull from Master's thesis for pringsheim's Theorem and all the good stuff. Don't forget uniqueness of eigenvectors for cone-preserving operators, this may be good to motivate the geometric elements.]

Define a \textbf{cone} $\K\subseteq\X$ to be a closed convex set that is \textbf{positively homogeneous} ($x\in\K$ and $\alpha\geq0$ implies that $\alpha x\in\K$), and \textbf{pointed} (if $x\in\K$ and $-x\in\K$, then $x=0$). We say $A\in\BX$ is a cone-preserving operator with respect to $\K$, or $\K$-preserving, if $A(\K)\subseteq\K$, i.e. $Ax\in\K$ for all $x\in\K$. When the cone $\K$ is clear from context, we simply say $A$ is cone-preserving. Note that a cone imposes a partial ordering, $``\geq,"$ on the Banach space $\X$ defined by $z\geq y$ for elements $y,z\in\X$ means $z-y\in\K$. In addition to assuming that $\X$ is a Banach space, throughout we will assume that we have fixed a cone $\K$ that orders our Banach space in this way. 

A cone is said to be \textbf{generating} if $\X = \K-\K$, i.e. for any $x\in\X$ there exists $y,z\in\K$ such that $x=y-z$.  It is easy to check that the set of cone-preserving linear operators forms a cone in $\BX$ when $\K$ is a generating cone on $\X$.  Similarly, if $A$ and $B$ are cone-preserving operators on $\X$ with respect to $\K$, we write $A\geq B$ (using a slight abuse of notation since it should clear whether our objects are vectors in $\X$ vs $\BX$) to mean that $A-B$ is cone-preserving with respect to $\K$. 

For a given cone $\K$, define the \textbf{dual cone} to be the set 
    $$ \KK = \{ f\in\XX \st f(x)\geq0 \text{ } \forall x\in\K\}. 
    $$
Despite terminology, $\KK$ may not satisfy the definitions to be a cone. For example, in $\R^2$ with the usual inner product, and the cone given as the nonnegative $x-$axis, $\K= \{(a,0)\in\R^2 \st a\geq0\}$, the dual cone is 
	$$\KK=\{y\in \R^2\st \langle x,y\rangle\geq0 \text{ } \forall x\in\K \} = \{(a,b) \st a\geq0\},
	$$
which is not pointed. However, whenever $\K$ is generating, $\KK$ is a cone. (Actually it is simple consequence of a Hahn-Banach separation theorem that $\KK$ is  a cone if and only if $\X= \overline{\K -\K}$, the topological closure of $\K-\K$. See chapter six in \cite{Deimling} for related discussion).

We say a cone is \textbf{normal} if there is a constant $\gamma>0$ such that whenever $0\leq x \leq y$, $||x||\leq \gamma||y||$. It is well known that a cone is normal if and only if its dual cone is generating, and similarly a cone is generating if and only if its dual cone is normal (see proposition 19.4 in \cite{Deimling}).

\textbf{Remark:} In Theorem \ref{lischneidertrich}, Li and Schneider consider the Banach space $\R^n$ ordered by the nonnegative orthant cone $\K=\{x=(x_1,x_2,..., x_n) \st x_i\geq0 \text{ } \forall i=1,2,...,n\}$. This cone is clearly generating: any vector in $\R^n$ is equal to the difference of two vectors with only nonnegative entries. This cone is also normal; in fact, any cone in a finite dimensional vectorspace is normal by the following lemma:

\begin{lemma}\label{finiteconesarenormal}(Lemma 1.2.5 in \cite{LemmensNussbaum}) Any cone defined on a finite dimensional, real normed vector space is normal. 
\end{lemma}
Thus, while we will later impose that the Banach space $\X$ is ordered by a generating and normal cone, this is not a restriction compared to Theorem \ref{lischneidertrich} because the nonnegative orthant cone meets both requirements.

\subsection{ Spectral Theory }

%[Talk about connection to stability as motivation]

%[talk about the complexification, maybe refer to Deimling]

% maybe include the statement that cone-preserving linear operators defined on all of \X are bounded in the appendix?

%For citations I think there was a good treatment of this in Engel and Nagel
Here we review some spectral theory. While our main results are stated for bounded operators, we do need some generalized notions for linear operators which are not necessarily bounded. We refer to chapter IV in \cite{EngelNagelSpectralTheory} for a more thorough discussion of this theory. For the following, let $A$ be a linear operator with dense domain $D(A)\subseteq \X$. $A$ is assumed to be a closed operator, meaning that the graph of $A$ is closed in $\X\times\X$. Note that any bounded operator on $\X$ is necessarily a closed operator by the closed graph theorem.

%complexificaiton 

%Just as when computing eigenvalues of a matrix acting on $R^n$, 
Though we are considering operators on a real Banach space, for various applications of spectral theory we consider the complexification of $(\X,||\cdot||)$, $\X_{\C}:=\X+i\X$, with norm given by 
	$$|| x+iy||_{\X_{\C}}=\sup_{\theta\in[0,2\pi]} ||\cos(\theta)x + \sin(\theta)y||,
	$$
for $x,y\in X$. Given a closed operator $A$ with domain $D(A)\subseteq \X$, we can then extend $A$ to its complexification $\tilde{A}$. The operator $\tilde{A}$ is a linear operator with domain $D(\tilde{A})=\{x+iy\st x,y\in D(A)\}$ dense in $\X_{\C}$. The operator $\tilde{A}$ is defined to act on $\X_{\C}$ by $\tilde{A}(x+iy)=Ax+iAy$, and is a closed operator when $A$ is. If $A:\X\to\X$ is bounded, then the norm on  $\X_{\C}$ is such that $||A||=||\tilde{A}||_{X_{\C}}$.  For the remainder of the section, we let $A$ implicitly refer to the complexification $\tilde{A}$. Define the \textbf{resolvent set} of an operator $A$ to be 
	$$\rho(A) := \{ \lambda \in\C \st  \lambda I -A \text{ is a bijection from $D(A)$ onto $X.$} \}.
	$$
For $\lambda\in\rho(A)$, $(\lambda I -A)^{-1}$ is a well defined linear operator on $X$, and is called the \textbf{resolvent of $A$ with respect to $\lambda$}. Moreover, since $A$ is a closed operator, $(\lambda I- A)^{-1}$ is in fact a bounded linear operator. We will frequently let $I$ be understood and instead write $(\lambda-A)^{-1}$. Define the \textbf{spectrum} of an operator $A$, $\sigma(A):= \C\setminus \rho(A)$.  %cite any functional analysis text

For any closed operator, define the \textbf{spectral abscissa} $s(A):= \sup\{ Re(\lambda) \st \lambda\in\sigma(A)\}$, where $Re(\lambda)$ is the real part of $\lambda\in\C$. Note that we allow $s(A)=\infty$ if this supremum is not finite, and $s(A)=-\infty$ if $\sigma(A)=\emptyset$. 

 If $A$ is a bounded linear operator, it is well known that $\sigma(A)$ is a nonempty compact subset of $\C$ (see for example Corollary 1.4 in chapter IV of \cite{EngelNagelSpectralTheory}). For any bounded linear operator $A$, we define the \textbf{spectral radius} of $A$ to be 
	$$r(A) := \sup\{ |\lambda| \st \lambda\in\sigma(A)\}.
	$$ 
Note that $r(A)$ is well defined whenever $A$ is bounded.

% A version of the spectral mapping thm for closed operators can be cited from Dunford and Schwartz. 
\begin{theorem}\label{spmapthm} (Spectral Mapping Theorem for Polynomials; see Theorem 7.4-2 in \cite{Kreyszig}) Let $A$ be a bounded linear operator on a complex Banach space $X$, and 
	$$ p(\lambda)  = \alpha_n \lambda^n + \alpha_{n-1} \lambda^{n-1} + \cdots + \alpha_0  
	$$
be a polynomial with $\alpha_n\neq0$. Then 
	$$\sigma(p(A)) = p(\sigma(A)),
	$$
i.e. the spectrum of $p(A)= \alpha_n A^n + \alpha_{n-1} A^n + \cdots + \alpha_0 I$ is $\{\alpha_n \mu^n + \alpha_{n-1} \mu^{n-1} + \cdots + \alpha_0  \st  \mu\in \sigma(A)\}$. 
\end{theorem}
The following useful corolary follows from the above spectral mapping theorem:
\begin{corollary}\label{spmapconsq} If $A$ is a bounded linear operator and $\alpha\geq0$, then $\alpha r(A)= r(\alpha A)$. 
\end{corollary}

In the case where $\X$ is a finite dimensional Banach space ordered by a generating and normal cone, the Perron-Frobenius theorem guarantees that the spectral radius of $A$ is an eigenvalue with corresponding eigenvector in $\K$. However, when $\X$ is infinite dimensional, $r(A)$ is not guaranteed to be an eigenvalue, but is guaranteed to be a point of the spectrum. 
% see if we can quote schaefer too? 
\begin{theorem}\label{spradinsp} (Exercise 11 in Chapter 6, Section 19, of \cite{Deimling}) \textit{Let $A$ be a bounded linear operator that preserves a generating and normal cone. Then $r(A)\in\sigma(A)$. }
\end{theorem}

We will also need the following comparison theorem:

\begin{theorem}\label{ordersprad}
(Ordering of Spectral Radii \cite{MarekIvo1970FToP}) Let $X$ be a Banach space with a generating and normal cone $K$. Let $A$ and $B$ be cone-preserving operators such that $B\leq A$. Then 
    $$ r(B) \leq r(A).
    $$
\end{theorem}

\subsection{ Resolvent Positive Operators } 

%Likely just omit this statement
%This is the relevant condition to ensure that the flow produced by the map $T$ under $\dot{x}=Tx$ preserves the cone $K$. 
%Here Think about showing the relevent semigroup condition? 
% Refer back to Arendt, Batty, others for theory on Resolvent Positive Operators

Let $X$ be a Banach space ordered by a cone $K$, and let $A$ be closed linear operator with dense domain in $\X$. $A$ is said to be \textbf{resolvent positive} if there is a ray $(\omega,\infty)\subset \rho(A)$ so that $(\lambda - A)^{-1}$ is  a cone-preserving operator for all $\lambda\in(\omega,\infty)$. 

If $A$ is the generator of a strongly continuous ($C_0$) semigroup, then the semigroup preserves the cone if $A$ is resolvent positive. This can be proved by invoking the following representation result for strongly continuous semigroups:

\begin{theorem} (Theorem 8.3 in chapter 1 of \cite{Pazy}) If $A$ is the infinitesimal generator of a $C_0$ semigroup $T(t)$ acting on a Banach space $\X$, then for all $t>0$ and all $x\in X$
	$$ T(t)x = \lim_{n\to\infty} \left[ \frac{n}{t}\left(\frac{n}{t} - A \right)^{-1}  \right]^n x .
	$$
This limit is uniform with respect to $t$ on any bounded interval in $(0,\infty)$. 
\end{theorem}
Note for any $t>0$, $\left(\frac{n}{t} - A \right)^{-1}$ is cone-preserving for all sufficiently large $n$ when $A$ is resolvent positive. Consequently, $\left[ \frac{n}{t}\left(\frac{n}{t} - A \right)^{-1}  \right]^n$ is cone-preserving for all sufficiently large $n$. Since $\K$ is closed, it follows that $T(t)x$ belongs to $\K$ when $x\in\K$.

We will next show that the converse holds as well; that is, cone-preserving $C_0$ semigroups have resolvent positive infinitesimal generators. To see why, assume $A$ is the infinitesimal generator of a $C_0$ semigroup $T(t)$. Then there exists $M\geq1$ and $\omega\in\R$ such that 
	$$  ||T(t)|| \leq Me^{\omega t}
	$$
for all $t\geq0$ (by Theorem 2.2 in \cite{Pazy}). 
Furthermore if $\lambda>\omega$, then $\lambda\in\rho(A)$, and we have the following representation formula for the resolvent of $A$ with respect to $\lambda$:
	$$(\lambda-A)^{-1}x = \int_0^\infty e^{-\lambda t}T(t)xdt
	$$
for all $x\in\X$ (see Theorem 5.3 in chapter 1 of \cite{Pazy} for a more details). It follows that whenever $T(t)$ is cone-preserving for all $t\geq0$, then $(\lambda-A)^{-1}x \in\K$ for all $x\in\K$ and all $\lambda>\omega$, using the fact that $\K$ is positively homogeneous and closed. Recall that the infinitesimal generator of a $C_0$ semigroup is a closed, densely defined operator. Therefore $A$ is resolvent positive when $T(t)$ is a cone-preserving semigroup.

Cone-preserving bounded linear operators are particularly simple examples of resolvent-positive operators. To see this, consider the representation of the resolvent of $A$ as a Neumann series: 
	$$ (\lambda-A)^{-1} = \sum_{n=0}^\infty \lambda^{-n-1}A^n,
	$$
which is valid for $\lambda>r(A)$.

\subsection{Convexity and Monotonicity of $r(F(\lambda-T^{-1})$}

Next we state a special case of Theorem 1.1 in \cite{RThiemeHorst1998Rorp}, regarding the behavior of the spectral radius map $\lambda\to r(F(\lambda-T)^{-1})$, which is tailored to our needs. 

% See if it is more fluid to give this statement, or to introduce C as a positive perturber of B. 
% I still have this issue about why Thieme only stated log convexity for the final case in his Theorem, when it is clear that it does not require $s(A)> s(B)$ for logconvexity to follow from Kato's Theorem applied to this problem. 

\begin{theorem}\label{thiemerorpthm} (Theorem 1.1 in \cite{RThiemeHorst1998Rorp}) Let $X$ be an ordered Banach space with a generating and normal cone $K$. Suppose $A$ is a resolvent positive operator with domain $D(A)$, and $A$ splits as $A=B+C$, where $B$ is a resolvent positive operator with domain $D(B)=D(A)$, $s(B)<0$, and $C$ maps $D(B)\bigcap K$ into $K$. Then $r(C(\lambda- B)^{-1})$ is a non-increasing, convex function of $\lambda\in(s(B),\infty)$, and exactly one of the following cases holds: 

1. $r(C(\lambda- B)^{-1})<1$ for all $\lambda > s(B)$. In this case, $s(A)=s(B)$.

2. There exist $\lambda_1 < \lambda_2$ in $(s(B),\infty)$ such that $r(C(\lambda_1- B)^{-1})\geq1$ and $r(C(\lambda_2- B)^{-1})<1$. In this case, $s(A)>s(B)$, and furthermore $r(C(s(A)I - B)^{-1})=1$. 
\end{theorem}

%Note that if $\hat{A}$ is a cone-preserving bounded linear operator that splits as $\hat{A}=T+F$, with $T$ and $F$ bounded, cone-preserving, and $r(T)<1$, then one may take $A=\hat{A}-I$, $B=T-I$, and $C=F$ in the statement of the previous Theorem. One can now check that $A$, $B$, and $C$ satisfy the conditions of Theorem \ref{thiemerorpthm}. We can then conclude that the map $\lambda\to r(F(\lambda - (T-I))^{-1})$ from $(r(T)-1,\infty)$ to $\R$ is a non-increasing and convex function of $\lambda$, and either 1 or 2 holds. Making the substitution $\mu= \lambda+1$, we have the same conclusion for the map $\mu\to r(F(\mu-T)^{-1})$ from $(r(T),\infty)$ to $\R$. Using this, proving Theorem \ref{thiemetrich} is a simple exercise in considering the three cases $R_0<1$, $R_0=1$, and $R_0>1$. 

Next we translate this result into a version which applies to bounded, cone-preserving linear operators.

\begin{corollary}\label{greggzrorpcor}  Let $X$ be an ordered Banach space with a generating and normal cone $K$. Suppose $\hat{A}$ is a bounded, cone-preserving linear operator on $\X$. Furthermore suppose that $\hat{A}$ splits as $\hat{A}=T+F$ where $T$ and $F$ are bounded, cone-preserving operators with $r(T)<1$. Then  $r(F(\lambda- T)^{-1})$ is a non-increasing, convex function of $\lambda\in(r(T),\infty)$, and exactly one of the following cases holds: 

1. $r(F(\lambda- T)^{-1})<1$ for all $\lambda > r(T)$. In this case, $r(\hat{A})=r(T)$.

2. There exists $\lambda_1 < \lambda_2$ in $(r(T),\infty)$ such that $r(F(\lambda_1- T)^{-1})\geq1$ and $r(F(\lambda_2- T)^{-1})<1$. In this case, $r(\hat{A})>r(T)$, and furthermore $r(F(r(\hat{A})I - T)^{-1})=1$. 
\end{corollary}

\textit{Proof.} We show that one may take $A=\hat{A}-I$ and $B=T-I$ with $D(A)=D(B)=\X$, and $C=F$ in the statement of Theorem \ref{thiemerorpthm}. Note that $s(T)=r(T)$, as $r(T)\in\sigma(T)$ by Theorem \ref{spradinsp}. Consequently, $s(B)=s(T-I)=s(T)-1$ by the spectral mapping theorem (Theorem \ref{spmapthm}). Hence $s(B)=r(T)-1$. Furthermore, if $\lambda>s(B)$, then $\lambda+1>r(T)$, and 
	$$(\lambda-B)^{-1}=(\lambda-(T-I))^{-1}=((\lambda+1)-T)^{-1}
	$$
is cone-preserving because $T$ is cone-preserving (and therefore resolvent positive). This shows that $B$ is resolvent positive. The same argument applied to $A=\hat{A}-I$ shows that $A$ is also resolvent positive. Note also that $s(B)=r(T)-1$ is negative since $r(T)<1$. Since $C=F$ is cone-preserving by assumption, $C$ maps $D(B)\bigcap \K=\K\bigcap\K=\K$ into $\K$. 

Thus we may apply Theorem \ref{thiemerorpthm}. Throughout we will further use the fact that $s(T-I)=r(T)-1$ and $s(\hat{A}-I)=r(\hat{A})-1$, and make the substitution $\mu=\lambda+1$. First, we conclude $r(F(\lambda-(T-I))^{-1})$ for $\lambda\in(s(T-I),\infty)$ is a non-increasing, convex function of $\lambda$, so equivalently $r(F(\mu-T)^{-1})$ for $\mu\in(r(T),\infty)$, is a non-increasing, convex function of $\mu$. 

In case 1, $r(F(\lambda- (T-I))^{-1})<1$ for all $\lambda > s(T-I)$ is equivalent to $r(F(\mu-T)^{-1})<1$ for all $\mu>r(T)$. In this case, $s(\hat{A}-I)=s(T-I)$ implies $r(\hat{A})-1=r(T)-1$, and $r(\hat{A})=r(T)$. 

In case 2, the existence of $\lambda_1 < \lambda_2$ in $(s(T-I),\infty)$ such that $r(F(\lambda_1- (T-I))^{-1})\geq1$ and $r(F(\lambda_2- (T-I))^{-1})<1$ is equivalent to the existence of $\mu_1 < \mu_2$ in $(r(T),\infty)$ such that $r(F(\mu_1-T)^{-1})\geq1$ and $r(F(\mu_2-T)^{-1})<1$. In this case, $s(\hat{A}-I)>s(T-I)$ implies $r(\hat{A})>r(T)$. Finally, $r(F[s(\hat{A}-I)I - (T-I)]^{-1})=1$ implies that 
	$$r(F[(r(\hat{A})-1)I +I -T]^{-1})=r(F(r(\hat{A})I-I+I-T)^{-1})=r(F(r(\hat{A})I-T)^{-1})=1. \text{ }\square
	$$

%
%We will treat the statement that $r(C(\lambda- B)^{-1})$ is a non-increasing, convex function, case 1, and case 2 separately. First, we conclude $r(F(\lambda-(T-I))^{-1})$, equivalently $r(F((\lambda+1)-T)^{-1})$ is a non-increasing, convex function of $\lambda\in(s(T-I),\infty)$. 
%
%
%1. $r(F(\lambda- (T-I))^{-1})<1$ for all $\lambda > s(T-I)$. In this case, $s(\hat{A}-I)=s(T-I)$.
%
%2. There exists $\lambda_1 < \lambda_2$ in $(s(T-I),\infty)$ such that $r(F(\lambda_1- (T-I))^{-1})\geq1$ and $r(F(\lambda_2- (T-I))^{-1})<1$. In this case, $s(\hat{A}-I)>s(T-I)$, and furthermore $r(F(s(\hat{A}-I)I - (T-I))^{-1})=1$. 
%

We will also need the following result:

\begin{lemma}\label{rorpprop} (Proposition 4.3 in \cite{RThiemeHorst1998Rorp}) Let $\X$, $\K$, $A$, $T$, and $F$, with $A=T+F$, satisfy the assumptions in Corollary \ref{greggzrorpcor}.  Let $\lambda\in\rho(T)$. Then

	$\lambda\in\rho(A)$ if and only if $1\in\rho( F(\lambda-T)^{-1} )$. Moreover in this case, 
		$$ (\lambda- A)^{-1} = (\lambda - T)^{-1}(I- F(\lambda-T)^{-1})^{-1}.
		$$
\end{lemma}

\textit{Proof.} This proposition is easily seen by simple algebra:
\begin{align*}
	\lambda- A  &= \lambda - T - F  \\
			&= I(\lambda-T) - F(\lambda-T)^{-1}(\lambda-T)\\
			&=  (I- F(\lambda-T)^{-1})(\lambda - T).  \\
\end{align*}
Then 
	$$ (\lambda- A)^{-1} =(\lambda - T)^{-1}(I- F(\lambda-T)^{-1})^{-1}
	$$
exactly when $\lambda\in\rho(A)$, and exactly when $1\in\rho(F(\lambda-T)^{-1})$. $\square$

\section{ Results }

% Do we need this, or is it sufficient that our spectral radius function is strictly decreasing at 1? 
%\subsection{ Lemma 1 Spectral Radius Function goes to 0 } 

% Make sure to use the term trichotomy to refer to Li \& Schneider's result so that the language here makes sense
To prove the trichotomy for $R_0$ in the matrix case, Li and Schneider first show that if $R_0>0$, then the spectral radius of $T+\frac{1}{R_0}F$ is 1. Here we show that in the infinite dimensional case, the assumptions of Theorem \ref{lischneidertrich} still guarantee that this property holds.

\subsection{$T+\frac{1}{R_0}F$ has spectral radius 1}

\begin{lemma}\label{spisone} Let $\X$ be a Banach space ordered by a generating and normal cone $\K$. Let $A$ be a bounded linear operator that preserves $\K$, and suppose $A$ is split as $A=T+F$, where $r(T)<1$ and $T$ and $F$ are also $\K$-preserving bounded linear operators. Define $R_0:= r(F(I-T)^{-1})$, and suppose $R_0>0$. Then $T+\frac{1}{R_0}F$ has spectral radius 1. 
\end{lemma}

\textit{Proof.} Let $A':= T+\frac{1}{R_0}F$. We begin by showing $r(A')\geq 1$. For $\lambda>r(T)$, note that 
    $$ (\lambda - A') = (I- \frac{1}{R_0}F(\lambda-T)^{-1})(\lambda - T)
    $$
and that $\lambda\in\rho(A') \iff 1\in\rho\left(\frac{1}{R_0}F(\lambda-T)^{-1}\right)$ from Lemma \ref{rorpprop}. Consequently, 
	\begin{align*}\lambda\in\sigma(A') \iff 1\in\sigma\left(\frac{1}{R_0}F(\lambda-T)^{-1}\right). \label{spectequiv}\tag{$*$}
	\end{align*}
Since $R_0=r\left(F(I-T)^{-1}\right)$ by definition, Corollary \ref{spmapconsq} and Theorem \ref{spradinsp} imply that 
	$$1=r\left(\frac{1}{R_0}F(I-T)^{-1}\right)\in\sigma\left(\frac{1}{R_0}F(I-T)^{-1}\right).
	$$ 
Note this is the case where $\lambda=1$ from (\ref{spectequiv}), and so it follows that $\lambda=1\in\sigma(A')$. This implies that $r(A')\geq1$.

For brevity define 
	$$\tilde{r}(\lambda):= r\left(\frac{1}{R_0}F(\lambda-T)^{-1}\right)
	$$ 
for $\lambda$ in the domain $(r(T),\infty)$ and recall $\tilde{r}(1)=1$ from above. Next we show $r(A')=1$ by contradicting the fact that $\tilde{r}(\lambda)$ is convex by Corollary \ref{greggzrorpcor}. Suppose that $r(A')>1$. Then there exists $\lambda_1>1$ such that $\lambda_1\in\sigma(A')$ (note we may always take $\lambda_1=r(A')\in\sigma(A')$ by applying Theorem \ref{spradinsp}). It follows that that $1\in\sigma\left(\frac{1}{R_0}F(\lambda_1-T)^{-1}\right)$ by (\ref{spectequiv}). It follows that $\tilde{r}(\lambda_1)\geq1$. However, by Corollary \ref{greggzrorpcor}, since $\tilde{r}(\lambda)$ is a non-increasing function, and since $\tilde{r}(1)=1$, it must be that $\tilde{r}(\lambda_1)$=1.

Note that we are necessarily in case 2 of Corollary \ref{greggzrorpcor}, and thus there exists $\lambda_2>\lambda_1$ with $\tilde{r}(\lambda_2)<1$. Let $t\in(0,1)$ such that $1t+\lambda_2(1-t)=\lambda_1$. Using convexity of $\tilde{r}(\lambda)$ and that fact $\tilde{r}(\lambda_2)<1$, we have the following contradiction:
\begin{align*}
	1 = \tilde{r}(\lambda_1) \leq t \tilde{r}(1) + (1-t)\tilde{r}(\lambda_2) = t + (1-t)\tilde{r}(\lambda_2)< 1. 
\end{align*}
Consequently, $r(A')$ cannot be greater than 1. Therefore $r(A')$ must be exactly 1.   $\square$

\subsection{ The Nonstrict Trichtomy }

We begin with the most general version of the trichotomy, with nonstrict inequalities.

\begin{theorem}\label{greggordtrich} The Nonstrict Trichotomy for Bounded Operators

% To remove the assumption that $R_0>0$, we simply treat this case separately. If $R_0$ is 0$ then we don't do our argument, we automatically have the bound. 

Let $\X$ be a Banach space ordered by a generating and normal cone $\K$. Let $A$ be a bounded linear operator that preserves $\K$, and splits as $A=T+F$, where $T$ and $F$ are also $\K$-preserving bounded linear operators with $r(T)<1$. Define $R_0:= r(F(I-T)^{-1})$. Then exactly one of the following holds: 

\begin{align*}
    (a) \text{ }R_0 \geq &\text{ } r(A) > 1, \\
    (b) \text{ }R_0 = & \text{ } r(A) = 1, \\
    (c) \text{ }R_0 \leq & \text{ } r(A) < 1.
\end{align*}
\end{theorem}

\textit{Proof.} 

%The proof combines the main ideas of Li and Schneider's proof of this result for nonnegative matrices and Thieme's trichotomy Theorem 3.6.2 for $R_0$ and $r(A)$. 

Let us start by assuming $R_0=0$. Then $r(A)<1$ by Theorem \ref{thiemetrich}, and we clearly have $R_0=0\leq r(A)$ as desired since the spectral radius is always nonnegative. 

Now assume $R_0>0$. By Lemma \ref{spisone}, we have that $r(T+\frac{1}{R_0}F) =1$. Consider 3 cases, $R_0=1$, $R_0>1$, and $0<R_0<1$. From Theorem \ref{thiemetrich}, we already have that when $R_0=1$, $r(A)=1$. All that is left is to show that $R_0\geq r(A)$ when $R_0>1$ and $R_0\leq r(A)$ when $R_0<1$. 

Suppose $R_0>1$. Then 
    $$ R_0(T+\frac{1}{R_0}F) \geq A
    $$
since $(R_0-1)T$ is cone preserving. By Theorem \ref{ordersprad}, it follows that  
    $$ r(R_0(T+\frac{1}{R_0}F)) \geq r(A).
    $$
From Corollary \ref{spmapconsq}, we have 
    $$ R_0 r(T+\frac{1}{R_0}F)\geq r(A).
    $$
Finally by Lemma \ref{spisone} we have that 
    $$ R_0 \geq r(A).
    $$

Conversely, if $0<R_0<1$, the same argument shows that 
    $$ R_0 \leq r(A).  \square
    $$

\subsection{The Strict Trichotomy } 

A version of Theorem \ref{greggordtrich} with strict inequalities can be obtained by imposing further geometric and spectral  assumptions on the operators $A$, $T$, and $F$. First we introduce a number of concepts and results from the literature which will allow us to state these assumptions.

A point $x\in\K$ is said to be an \textbf{almost interior point} if $f(x)>0$ for all $f\in\KK\setminus\{0\}$; see \cite{GlückJochen2020Aipi}. Simiarly, a \textbf{strictly positive functional} $f\in\KK$ is a functional that satisfies $f(x)>0$ whenever $x\in\K\setminus\{0\}$. An bounded, cone-preserving operator $A$ is called \textbf{semi-nonsupporting} if for any pair $(x,f)\in(\K\setminus\{0\})\times(\KK\setminus\{0\})$, there is an $n=n(x,f)$ (which may depend on $x$ and $f$) such that $f(A^nx)>0$. Semi-nonsupporting operators generalize the notion of irreducible matrices. For these operators, we have the following:

\begin{theorem}\label{MarekPFthm} (Theorem 2.2 in \cite{MarekIvo1970FToP} ) Let $\X$ be a Banach space ordered by a generating and normal cone $\K$. Let $A$ be a bounded cone-preserving, semi-nonsupporting operator on a generating and normal cone, and suppose $r(A)$ is a pole of the resolvent map of $A$ (the map from $\rho(A)$ to $\BX$ given by $\lambda\to (\lambda-A)^{-1}$). Then 

\begin{adjustwidth}{1cm}{1cm}
1. $r(A)$ is a simple pole of the resolvent map. \\ 
2. There is a unique eigenvector $x$ in $\K$ (up to positive scaling), which is an almost interior point, corresponding to $r(A)$. Furthermore, if $v\in\K$ is an eigenvector of $A$, $v=cx$ for some positive constant $c$. \\
3. There is a strictly positive eigenfunctional $f$ of $A$ corresponding to $r(A)$.
\end{adjustwidth}
\end{theorem}

The condition that $r(A)$ is a pole of the resolvent map guarantees that $r(A)$ will be an eigenvalue of $A$ (see for example Corollary 6.12 in chapter VII of \cite{conway}). A sufficient condition for $r(A)$ to be a pole is to assume that $A$ is a cone-preserving compact operator on $\X$ with $r(A)>0$ (Theorem VII.4.5 in \cite{DunfordSchwartz} guarantees that nonzero spectral values of a compact operator are poles, and Theorem \ref{spradinsp} guarantees $r(A)$ is one of those values). 

The following is a version of Theorem \ref{ordersprad} which gives the strict ordering of spectral radii under additional assumptions. 

\begin{theorem}\label{strictordersp} (Theorem 4.3 in \cite{MarekIvo1970FToP}) Let $\X$ be a Banach space ordered by a generating and normal cone $\K$. Let $A$ and $B$ be bounded cone-preserving operators such that $B\leq A$ and $B\neq A$. Furthermore assume that $A$ is a semi-nonsupporting operator with eigenvector $x\in\K$ and eigenfunctional $f\in\KK$ both corresponding to $r(A)$, and assume that $g\in\KK$ is an eigenfunctional of $B$ corresponding to $r(B)$. Then
    $$ r(B) < r(A).
    $$
\end{theorem}

The following lemma allows us to show that if $R_0>0$, then $T+\frac{1}{R_0}F$ is semi-nonsupporting whenever $A$ is. 

\begin{lemma}\label{seminonsupopsareord} Let $\X$ be a Banach space ordered by a cone $\K$. Suppose $B$ and $C$ are bounded cone-preserving linear operators with $B\leq C$. If $B$ is semi-nonsupporting, then so is $C$. 
\end{lemma}

\textit{Proof.} We begin by showing that $B^n\leq C^n$ for any $n\in\N$ whenever $0\leq B\leq C$. Suppose integers $i,j$ are such that $i\geq1$, $j\geq0$, and $i+j=n$. Then $B^iC^j \leq B^{i-1}C^{j+1}$ follows from
	$$ B^{i-1} C^{j+1}- B^iC^j = (B^{i-1}C-B^i)C^j = B^{i-1}(C-B)C^j,
	$$
because this expression is a composition of cone-preserving operators. Then 
	$$ B^n \leq B^{n-1}C \leq B^{n-2}C^2 \leq \cdots \leq B C^{n-1} \leq C^n. 
	$$
Now suppose $B$ is semi-nonsupporting. Then for any any pair $(x,f)\in(\K\setminus\{0\})\times(\KK\setminus\{0\})$, there is an $n=n(x,f)$ such that $f(B^n x)>0$. Clearly for this $x, f$, and the same value of $n$, we have 
	$$ f(C^n x) \geq f(B^n x) > 0,
	$$
showing that $C$ is semi-nonsupporting as well. $\square$

\begin{corollary}\label{seminonsuppsplit} Let $\X$ be a Banach space ordered by a cone $\K$. Suppose $A$ is a bounded, cone-preserving, semi-nonsupporting operator that splits as $A=T+F$, with $T$ and $F$ bounded, cone-preserving operators. Then $aT+ bF$ is semi-nonsupporting for any $a,b>0$. 
\end{corollary}

\textit{Proof.} Let $\alpha=\min\{a,b\}$. Then $aT+bF\geq \alpha(T+F)=\alpha A$. Clearly $\alpha A$ is semi-nonsupporting as $f((\alpha A)^n x)>0$ whenever $f(A^n x)>0$,  $(x,f)\in(\K\setminus\{0\})\times(\KK\setminus\{0\})$. Thus by Lemma
\ref{seminonsupopsareord}, $aT+bF$ is semi-nonsupporting. $\square$

\begin{theorem}\label{greggstrictordtrich}  The Strict Trichotomy for Bounded Operators

Let $\X$ be a Banach space ordered by a generating and normal cone $\K$. Let $A$ be a bounded, cone-preserving, semi-nonsupporting, linear operator whose spectral radius is a pole of its resolvent map. Assume that $A$ splits as $A=T+F$, for bounded cone-preserving linear operators $T$ and $F$ with $r(T)<1$. Furthermore, suppose $T$ is not the zero operator. Define $R_0 := r(F(I-T)^{-1})$ and assume $R_0>0$. Suppose that $F(I-T)^{-1}$ has $R_0$ as a pole of its resolvent. Then exactly one of the following holds: 
    \begin{align*}
    (a) \text{ }R_0 > & \text{ } r(A) > 1, \\
    (b) \text{ }R_0 = & \text{ } r(A) = 1, \\
    (c) \text{ }R_0 < & \text{ } r(A) < 1.
\end{align*}
\end{theorem}

%To make this Theorem work, we need to have that 
%    \begin{align*}
%        R_0 ( T + \frac{1}{R_0}F ) &\geq A \\
%        R_0 ( T + \frac{1}{R_0}F ) &\leq A 
%    \end{align*}
%with the requirements on Marek 4.3, i.e. both $A$ and $T+\frac{1}{R_0}F$ need to be semi-nonsupporting, which will follow if we assume $A$ is semi-nonsupporting. Furthermore, we need their spectral radii to be poles of the resolvent, meaning that we need $r(A)$ to be an eigenvalue with eigenvector in the cone and the same for $T+\frac{1}{R_0}F$. Which means that we need to assume that the spectral radius is a pole of the resolvent, such as when the operator is compact (or quasi-compact). 
%

\textit{Proof.} 

After applying Theorem \ref{greggordtrich}, it remains only to be shown that $R_0>r(A)$ when $R_0\geq r(A)$ and $R_0<r(A)$ when $R_0\leq r(A)$. 

By Theorem \ref{MarekPFthm}, $R_0$ is an eigenvalue of $F(I-T)^{-1}$ with eigenvector $v$ and eigenfunctional $f$ in the cone and dual cone, repsectively. Next it is shown that $T+\frac{1}{R_0}F$ also has eigenvector and eigenfunctional in the cone and dual cone corresponding to its spectral radius. Recall from Lemma \ref{spisone} that $r(T+\frac{1}{R_0}F)=1$.  First note that $f$ is also the eigenfunctional corresponding to $r\left(T+\frac{1}{R_0}F\right)=1$ by the following:
    \begin{align*}
        [F(I-T)^{-1}]^* f &= R_0 f \\
      \implies \hspace{28pt} \frac{1}{R_0}F^* f &= (I-T^*) f \\
     \implies   (T + \frac{1}{R_0}F)^* f &=  f. \\
    \end{align*}
Note that if $v$ is an eigenvector of $F(I-T)^{-1}$ corresponding to $R_0$, then setting $y=(I-T)^{-1}v$ and noting that $(I-T)^{-1}Fy= R_0y$ implies that $(T+\frac{1}{R_0}F)y=y$. This shows that $y\in\K\setminus\{0\}$ is an eigenvector of $T+\frac{1}{R_0}F$ corresponding to its spectral radius, 1.

Note that since $A$ has its spectral radius as a pole of the resolvent, Theorem \ref{MarekPFthm} implies that $r(A)$ is an eigenvalue of $A$ with corresponding eigenvector and eigenfunctional in the cone and dual cone, respectively.

Consider the case where $R_0>1$. Then $R_0(T+\frac{1}{R_0}F)\geq A$. The operator $A$ is assumed to be semi-nonsupporting, and so by Corollary \ref{seminonsuppsplit}, $R_0(T+\frac{1}{R_0}F)$ is semi-nonsupporting. Also clearly $R_0(T+\frac{1}{R_0}F)\neq A$ since $R_0\neq1$. We have shown that $T+\frac{1}{R_0}F$, hence $R_0(T+\frac{1}{R_0}F)$, has an eigenvector in $\K$ and an eigenfunctional in $\KK$ corresponding to its spectral radius. Similarly by Theorem \ref{MarekPFthm}, $A$ has an eigenfunctional in $\KK$ corresponding to $r(A)$. Thus by Lemma \ref{spisone}, Corollary \ref{spmapconsq}, and Theorem \ref{strictordersp},
	$$ R_0 = R_0r(T+\frac{1}{R_0}F) = r(R_0(T+\frac{1}{R_0}F)) > r(A).
	$$

Consider the case where $0<R_0<1$. Then $R_0(T+\frac{1}{R_0}F)\leq A$. Again clearly $R_0(T+\frac{1}{R_0}F)\neq A$ since $R_0\neq1$. We have shown that $T+\frac{1}{R_0}F$, hence $R_0(T+\frac{1}{R_0}F)$, has an eigenfunctional in $\KK$ corresponding to its spectral radius. Similarly by Theorem \ref{MarekPFthm}, $A$ has an eigenvector in $\K$ and an eigenfunctional in $\KK$ corresponding to $r(A)$. Finally, again by Lemma \ref{spisone}, Corollary \ref{spmapconsq}, and Theorem \ref{strictordersp},
	$$ R_0 = R_0r(T+\frac{1}{R_0}F) = r(R_0(T+\frac{1}{R_0}F)) < r(A). \text{ } \square
	$$

\section{ Example:  Leslie Models of Population Growth on $\ell_p$}

Consider an age-structured population model with a countably infinite number of discrete age classes. The state of the population at time step $t$ is represented by $x(t)= (x_i(t))\in\ell_p$ for $p\in[1,\infty]$, $i\in\N$, $t\in\N\bigcup\{0\}$, where each component $x_i(t)$ is the population size of age class $i$ at time $t$. To construct a Leslie model in this setting, we consider a dynamical system of the form $x(t+1)= Ax(t)$, where $A=T+F$ is the combination of a a survival operator $T$ and a fertility operator $F$, defined in the following. 

Let $f=(f_i)\in \ell_q$ ($q$ such that $\frac{1}{p}+\frac{1}{q}=1$) with $f_i\geq0$ for all $i$. Each $f_i$ represents the expected number of offspring in a single time step per individual in age class $i$. For $x=(x_i)$, define $Fx:= (\sum_{i=1}^\infty f_ix_i,0,0,0,....)$. Define a sequence $(t_i)$ such that each $t_i$ represents the expected fraction of individuals from age class $i$ to survive to age class $i+1$ in a single time step. It is assumed $0\leq t_i\leq 1$ for all $i\in\N$ and $\limsup_{i\to\infty} t_i <1$. The latter condition will be used to show that $r(T)<1$. For $x=(x_i)$, define $Tx=(0,t_1x_1, t_2x_2, t_3x_3,...)$.

The action of $A=T+F$ on an element $x=(x_i)\in\ell_p$ can be represented by an infinite dimensional matrix:
	$$ \begin{bmatrix} f_1 & f_2 & f_3 & f_4 & ... \\
				t_1 & 0 & 0 & 0 & .... \\
				0 & t_2 & 0 & 0 & .... \\
				0 & 0 & t_3 & 0 & .... \\
				\vdots & \vdots & \vdots & \ddots & .... \\
\end{bmatrix}
\begin{bmatrix} x_1 \\
x_2 \\
x_3 \\
x_4 \\
\vdots \\
\end{bmatrix} = 
\begin{bmatrix} \sum_{i=1}^\infty f_ix_i \\
t_1 x_1 \\ 
t_2x_2 \\
t_3x_3 \\
\vdots \\
\end{bmatrix}.
	$$

$T$ and $F$ (and therefore $A$) preserve the standard cone in $\ell_p$, which is 
	$$K=\ell_p^+:=\{x\in\ell_p \st x_i\geq0 \text{ }\forall i\in\N\}.
	$$
This cone is clearly generating and normal with normality constant $\gamma=1$. 

\textbf{Claim 1:} $T:\ell_p\to\ell_p$ is a $\K$-preserving, bounded linear operator, and $r(T)<1$. 

\textit{Proof.} The fact that $T$ is cone-preserving is obvious since $0\leq t_i\leq 1$. Also since $0\leq t_i\leq 1$, $|t_ix_i|\leq |x_i|$ for all $i\in\N$. Therefore $||Tx||_p\leq ||x||_p$ for all $x\in\ell_p$, and so $T$ is bounded.

Since $\limsup_{i\to\infty} t_i <1$, there is some $m\in\N$ and $\varepsilon>0$ such that $t_i<1-\varepsilon$ for all $i> m$. The sequence $(t_i)$ is component-wise bounded above by a sequence $(s_i)$ defined by  \\
$s_i:=\begin{cases} 1 & \text{ if } i\leq m \\
1-\varepsilon & \text{ otherwise} \\
\end{cases}.$
Define an operator $S$ by $Sx= (0, s_1x_1,s_2x_2,s_3x_3,....)$. It is clear that $S\geq T\geq0$, so by Theorem \ref{ordersprad} if we show that $r(S)<1$, then $r(T)<1$. Considering Gelfand's formula for the spectral radius,
	$$ r(S)= \lim_{n\to\infty} ||S^n||^\frac{1}{n},
	$$
we fix $x\in\ell_p$ with $||x||_p=1$. Then for $n> m$, 
\begin{align*}
	||S^n x||_p &\leq (1-\varepsilon)^{n-m}||x||_p =  (1-\varepsilon)^{n-m}.    \\
\end{align*}
Then
\begin{align*}
	||S^n|| = \sup_{||x||_p=1} ||S^n x||_p &\leq (1-\varepsilon)^{n-m} ,  \\
	||S^n||^\frac{1}{n} &\leq (1-\varepsilon)^\frac{n-m}{n}, \\
	r(S)= \lim_{n\to\infty} ||S^n||^\frac{1}{n} &\leq \lim_{n\to\infty} (1-\varepsilon)^\frac{n-m}{n} = (1-\varepsilon) < 1. \square \\
\end{align*}	

\textbf{Claim 2:} $F$ is a $\K$-preserving bounded linear operator. 

\textit{Proof.} It is obvious that $F$ is $\K$-preserving since $f_i\geq0$ for all $i\in\N$. To see that $F$ is bounded, note that for $x=(x_i)\in\ell_p$, 
	$$||Fx||_p = \left|\sum_{i=1}^\infty f_ix_i \right| \leq  \sum_{i=1}^\infty |f_ix_i|  \leq ||f||_q ||x||_p
	$$ 
by H\"{o}lder's inequality. $\square$

% What conditions must we impose on A,T and F to get strict inequality?
Since all of the conditions of Theorem \ref{greggordtrich} are satisfied, we may define $R_0:= r(F(I-T)^{-1})$. Writing $(I-T)^{-1}$ as an infinite matrix yields
	$$ (I-T)^{-1} = \sum_{n=0}^\infty T^n = 
\begin{bmatrix} 
1 & 0 & 0 & 0 & 0 &... \\
t_1 & 1 & 0 & 0 & 0 &  ... \\
t_1t_2 & t_2 & 1 & 0 & 0 & ...\\
t_1t_2t_3 & t_2t_3 & t_3 & 1 & 0 & ... \\
t_1t_2t_3t_4 & t_2t_3t_4 & t_3t_4 &t_4 & 1 & ...\\   
\vdots &  \vdots & \vdots & \vdots & \ddots &\ddots \\
\end{bmatrix},
	$$
and it follows that
$$F(I-T)^{-1}=
\begin{bmatrix}
f_1+\sum_{i=2}^\infty(f_i\prod_{j=1}^{i-1}t_j)  & f_2+\sum_{i=3}^\infty(f_i\prod_{j=2}^{i-1}t_j) &  f_3+\sum_{i=4}^\infty(f_i\prod_{j=3}^{i-1}t_j) & ... \\
0 & 0 & 0 & ... \\
0 & 0 & 0 & ... \\
0 & 0 & 0 & ... \\
\vdots & \vdots & \vdots & ... \\
\end{bmatrix}.
$$

Since $F(I-T)^{-1}$ is a rank 1 operator with unique normalized eigenvector $(1,0,0,0,...)$, it is easy to see that $R_0$ is given by the expression $f_1+\sum_{i=2}^\infty(f_i\prod_{j=1}^{i-1}t_j)$. By Theorem \ref{greggordtrich}, this expression gives a bound for how far $r(A)$ is from 1, and determines whether $r(A)$ is greater than, less than, or equal to $1$.

% lim inf doesn't work
%Since $\liminf_{i\to\infty} t_i <1$, we can fix a subsequence $(t_{i_k})$ such that $t_{i_k}< \delta <1$ for all $k\in\N$. Define a new sequence $(s_j)$ by 
%	$$s_j=\begin{cases} \delta & \text{ if } j=i_k \text{ for some $k\in\K$, i.e. if $j$ indexes an element of our subsequence of $t_i$.} \\
%1 & \text{ otherwise,} \\
%\end{cases}
%	$$
%and define an operator $S$ by $Sx= (0, s_1x_1,s_2x_2,s_3x_3,....)$. It is clear that $S\geq T\geq0$, so by Theorem \ref{ordersprad} if we show that $r(S)<1$, then $r(T)<1$. Using Gelfand's formula for the spectral radius: 
%	$$ r(S)= \lim_{n\to\infty} ||S^n||^\frac{1}{n},
%	$$
%Fix $x\in\ell_p$ with $||x||_p=1$. Then $||S^n x||^\frac{1}{n} \leq 
%

%\printbibliography
%
\bibliographystyle{plain}
\bibliography{Primo.bib}

%\section{ References } 

%  R. Thieme, H. (1998). Remarks on resolvent positive operators and their perturbation. Discrete and Continuous Dynamical Systems. Series A, 4(1), 73–90. https://doi.org/10.3934/dcds.1998.4.73

\end{document}